\documentclass[a4paper,openbib,12pt]{article}
\usepackage{geometry}
\usepackage{amsmath}
\usepackage{amssymb}
\usepackage{amsfonts}

\newtheorem{theorem}{Theorem}[section]

\newtheorem{corollary}[theorem]{Corollary}

\newtheorem{lemma}[theorem]{Lemma}

\newtheorem{remark}[theorem]{Remark}

\newenvironment{proof}[1][Proof]{\noindent\textbf{#1.} }{\ \rule{0.5em}{0.5em}}
\input{tcilatex}
\begin{document}

\title{Topological symmetries of simply-connected four-manifolds and actions
of automorphism groups of free groups}
\author{Shengkui Ye}
\maketitle

\begin{abstract}
Let $M$ be a simply connected closed $4$-manifold. It is proved that any
(possibly finite) compact Lie group acting effectively and homologically
trivially on $M$ by homeomorphisms is an abelian group of rank at most two,
when \ $b_{2}(M)>2$. As applications, let $\mathrm{Aut}(F_{n})$ be the
automorphism group of the free group of rank $n.$ We prove that any group
action of $\mathrm{Aut}(F_{n})$ (and thus $\mathrm{GL}_{n}(\mathbb{Z})$) $%
(n\geq 4)$ on $M\neq S^{4}$ by homologically trivial homeomorphisms factors
through $\mathbb{Z}/2.$
\end{abstract}

\section{Introduction}

Let $F_{n}$ be a free group of rank $n$ and $\mathrm{Aut}(F_{n})$ the
automorphism group, $\mathrm{SAut}(F_{n})$ its unique index-two subgroup. It
is generally believed that any group action of $\mathrm{Aut}(F_{n})$ on a
compact manifold $N^{k}$ factors a finite group when $k<n-1.$ This is a
general version of the Zimmer program, studying actions of irreducible
high-rank lattices on manifolds (cf. \cite{fi}). Bridson and Vogtmann \cite%
{bv} prove that any action of $\mathrm{SAut}(F_{n})$ on a sphere $S^{k}$ $%
(k<n-1)$ is trivial. Let $M$ be an orientable manifold of Euler
characteristic not divisible by $6.$ The author \cite{ye182} proves that any
action of $\mathrm{SAut}(F_{n})$ on $M^{k}$ $(k<n-1)$ is trivial. In this
article, we study actions of $\mathrm{SAut}(F_{n})$ on simply connected
4-manifolds. Note that the abelianization $F_{n}\rightarrow \mathbb{Z}^{n}$
induces epimorphisms $\mathrm{Aut}(F_{n})\rightarrow \mathrm{GL}_{n}(\mathbb{%
Z})$ and $\mathrm{SAut}(F_{n})\rightarrow \mathrm{SL}_{n}(\mathbb{Z}).$ We
prove the following.

\begin{theorem}
\label{th2}Let $\mathrm{SAut}(F_{n})$ be the unique index-two subgroup of
the automorphism group of the free group $F_{n}$ and $M$ a closed 4-manifold
with $H_{1}(M;\mathbb{Z})=0$. When $b_{2}(M)\geq 1,$ any group action of $%
\mathrm{SAut}(F_{n})$ (and thus $\mathrm{SL}_{n}(\mathbb{Z})$) $(n\geq 4)$
on $M$ by homologically trivial homeomorphisms is trivial$.$
\end{theorem}

\begin{remark}
\begin{enumerate}
\item[(i)] Combining Theorem \ref{th2} with a result of Bridson-Vogtmann 
\cite{bv}, it is true that any group action of $\mathrm{SAut}(F_{n})$ $%
(n\geq 6)$ on any simply connected 4-manifold $M$ by homologically trivial
homeomorphisms is trivial$.$ Without restrictions of homological triviality,
the group $\mathrm{SAut}(F_{n})$ could act non-trivially on some $M$ through
a finite quotient group.

\item[(ii)] When $n=3,$ the group $\mathrm{SAut}(F_{3})$ could act through $%
\mathrm{SL}_{3}(\mathbb{Z})$ on $\mathbb{C}P^{2}$ and $S^{2}$ (thus $%
S^{2}\times S^{2}$) non-trivially. This means that the inequality $n\geq 4$
in Theorem \ref{th2} cannot be improved. It is interesting to notice that
the critical size $n=4$ and the dimension of $M$ (excluding $S^{4}$) are the
same, while it is generally required that $n>\dim M+1$ in Zimmer's program.
\end{enumerate}
\end{remark}

The proof of Theorem \ref{th2} is based on the study of symmetries of
4-manifolds. Let $M$ be a simply connected closed topological 4-manifold.
Such a manifold was classified by Freedman in terms of the intersection
pairing on second homology and the Kirby--Siebenmann invariant. As a first
step, we study the topological transformation groups on $M.$ Since any
finite group $G$ acts freely on a simply connected closed 4-manifold $M$
(the universal cover of a 4-manifold with fundamental group $G$) and acts
effectively on the homology $H_{\ast }(M;\mathbb{Z})$ by the Lefschetz
fixed-point theorem, we restrict our attention to the homologically trivial
actions. It turns out that these actions are very restrictive. The locally
linear version of the following result was firstly proved by McCooey \cite%
{mc}.

\begin{theorem}
\label{main}Let $G$ be a (possibly finite) compact Lie group acting
effectively and homologically trivially on a topological closed 4-manifold $%
M $ with $H_{1}(M;\mathbb{Z})=0.$ When the second Betti number $b_{2}(M)\geq
3, $ the group $G$ is a subgroup of $S^{1}\times S^{1}.$
\end{theorem}

\begin{remark}

\begin{enumerate}
\item[(i)] The dihedral group $D_{n}$ could act effectively and
homologically trivially on $S^{2}$ and thus on $S^{2}\times S^{2}.$ This
means that the bound of Betti numbers in Theorem \ref{main} is sharp.

\item[(ii)] Let the cyclic group $\mathbb{Z}/p$ of order prime $p$ act on $%
S^{2}$ by rotations. The product $\mathbb{Z}/p\times \mathbb{Z}/p$ acts on $%
S^{2}\times S^{2}$ by rotating each factor with fixed points $S^{0}\times
S^{0}$. Taking connected sum along an invariant ball around a fixed point,
the product $\mathbb{Z}/p\times \mathbb{Z}/p$ could act homologically
trivially on the connected sum $(S^{2}\times S^{2})\#(S^{2}\times S^{2}).$
Therefore, the bound of ranks of the Lie group $G$ in Theorem \ref{main} is
sharp.
\end{enumerate}
\end{remark}

The actions of finite groups on 4-manifolds are already studied by many
people, eg. Edmonds \cite{ad,ad2}, Hambleton and Lee \cite{hl,hl2}, McCooey 
\cite{mc,mc2} and Wilczy\'{n}ski \cite{w}, among others. For most of these
works, the group actions are assumed to be locally linear (or smooth). The
actions considered in this article are topological. Without locally linear
assumptions, we cannot use the real representation of isotropy subgroups of
a fixed point (see the first paragraph of Section \ref{sect} for an explicit
list of difficulties). We will use the Smith theory of homology manifolds,
equivariant cohomology, localization method and Borel formulas to deal with
this difficulty.

The paper is organized as the following. In Section 2, we collect some basic
facts and lemmas, most of which are already known in the literature. In
Section 3, we study the group action of minimal non-abelian finite groups on
simply connected 4-manifolds and Theorem \ref{main} is proved. In the last
section, we prove Theorem \ref{th2}.

\section{Preliminary}

Let $L=\mathbb{Z}$ or $F_{p},$ the finite field with prime $p$ elements. All
homology groups are Borel-Moore homology groups with compact supports and
coefficients in a sheaf $\mathcal{A}$ of modules over $L$. The homology
groups of $X$ are denoted by $H_{\ast }(X;\mathcal{A})$ and the cohomology
groups (with coefficients in $\mathcal{A}$ and compact supports) are denoted
by $H^{\ast }(X;\mathcal{A}).$ If $\mathcal{A}$ is the constant sheaf, this
is isomorphic to the \v{C}ech cohomology with compact supports. If $F$ is a
closed subset of $X$, then sheaf cohomology satisfies $H^{k}(X,F;\mathcal{A}%
)\cong H^{k}(X-F;\mathcal{A}).$ The (co)homology $n$-manifold over $L$
considered in this article will be as in Borel \cite{Bo}. Roughly speaking,
a homology $n$-manifold over $L$ (denoted by $n$-hm$_{L}$) is a locally
compact Hausdorff space that has a local cohomology structure (with
coefficient group $L$) resembling that of the Euclidean $n$-space.
Topological manifolds are (co)homology manifolds over $L.$ Homology
manifolds satisfy Poincar\'{e} duality between Borel-Moore homology and
sheaf cohomology (\cite{br}, Thm 9.2, p.329)$.$

We need several lemmas. The following result is generally called the Local
Smith Theorem (cf. \cite{br} Theorem 20.1, Prop 20.2, pp. 409-410).

\begin{lemma}
\label{sm}Let $p$ be a prime and $L=F_{p}$. The fixed point set of any
action of the cyclic group $\mathbb{Z}/p$ of order $p$ on an $n$-hm$_{L}$ is
the disjoint union of (open and closed) components each of which is an $r$-hm%
$_{L}$ with $r\leq n$. If $p$ is odd then each component of the fixed point
set has even codimension.
\end{lemma}

The following lemma is from Bredon \cite{br}, Theorem 2.5, p.79.

\begin{lemma}
\label{Bredon}Let $G$ be a group of order $2$ operating effectively on an $n$%
-cm over $\mathbb{Z}$, with non-empty fixed points. Let $F_{0}$ be a
connected component of the fixed point set of $G$, and $r=\dim _{2}F_{0}$.
Then $n-r$ is even (respectively odd) if and only if $G$ preserves
(respectively reverses) the local orientation around some point of $F_{0}.$
\end{lemma}

The following lemma is from Bredon \cite{br}, Theorem 16.32, p.388.

\begin{lemma}
\label{tmf}If $X$ is a second countable $n$-hm$_{L},$ with or without
boundary, and $n\leq 2,$ then $X$ is a topological $n$-manifold.
\end{lemma}

Let a finite group $G\ $act on a space (usually an $n$-hm$_{L}$) $X.$ The
group $G$ acts on the product $X\times EG$ diagonally, where $EG$ is the
total space of a classifying space $BG$. The Borel construction is the
quotient space of $X\times EG,$ denoted by $X_{G}.$ The Leray-Serre spectral
sequence for the fibration $X\rightarrow X_{G}\rightarrow BG$ is 
\begin{equation*}
H^{i}(G;H^{j}(X))\Longrightarrow H^{i+j}(X_{G}),
\end{equation*}%
which is also called the Borel spectral sequence. Let $R$ be a PID. When $G$
acts trivially on $H^{\ast }(X;R)$ (which is a finitely generated free $R$%
-module) and the Borel spectral sequence degenerates, we have $H^{\ast
}(X_{G};R)\cong H^{\ast }(B_{G};R)\bigotimes H^{\ast }(X;R)$ as graded
modules (cf. \cite{tom}, Proposition 1.18 of Chapter III, page 182). Let $%
\Sigma =\{x\in X\mid $ the stabilizer $G_{x}\neq 1\}$ be the singular set.
Let $S\supseteq \{1\}$ be a multiplicative set of $H^{\ast }(G)$ and $%
X^{S}=\{x\in X\mid S\cap \ker (H^{\ast }(G)\rightarrow H^{\ast
}(G_{x}))=\emptyset \}.$ The localization theorem says that the restricted
homomorphism 
\begin{equation*}
S^{-1}H^{\ast }(X_{G})\rightarrow S^{-1}H^{\ast }(X_{G}^{S})
\end{equation*}%
is an isomorphism (cf. Hsiang \cite{xiang}, p.40).

We first prove the following collapsing of spectral sequences.

\begin{lemma}
\label{collp}Let $G$ be a finite group acting on a manifold $M^{2n}$ with
vanishing odd-dimensional \emph{integral} cohomology groups and torsion-free
even dimensional integral cohomology groups. Suppose that the group action
is homologically trivial and $H^{i}(G;\mathbb{Z})=0$ for each odd $i.$ Then
the Borel spectral sequence $H^{i}(BG;H^{j}(M;\mathbb{Z}))\Longrightarrow
H^{i+j}(M_{G};\mathbb{Z})$ collapses.
\end{lemma}

\begin{proof}
The differential map 
\begin{equation*}
d_{s}:H^{i}(BG;H^{j}(M;\mathbb{Z}))\rightarrow H^{i+s+1}(BG;H^{j-s}(M;%
\mathbb{Z}))
\end{equation*}%
is trivial when $s$ is even, since the odd dimensional cohomology group of $%
G $ is trivial. When $s$ is odd, the map $d_{s}$ is still trivial since
either $H^{j}(M;\mathbb{Z})$ or $H^{j-s}(M;\mathbb{Z})$ is trivial.
\end{proof}

The following is essentially Cor. 9.3 of Bredon \cite{br} (p.249).

\begin{lemma}
\label{bred}The inclusion $\Sigma \rightarrow M$ induces isomorphism $%
H^{i}(M_{G};\mathcal{A})\rightarrow H^{i}(\Sigma _{G};\mathcal{A})$ for any
coefficients $\mathcal{A}$ and $i>n.$
\end{lemma}

We will need the following lemma, part of whose proof is already contained
in Prop. 2.4 of Edmonds \cite{ad}.

\begin{lemma}
\label{first}Let $G=\mathbb{Z}/p$ be a cyclic group of a prime order $p$
acting effectively and homologically trivially on a simply connected closed
topological 4-manifold $M.$ The fixed point set $M^{G}$ is a disjoint union
of spheres $S^{2}$ and discrete points.
\end{lemma}

\begin{proof}
By the Smith theory, the fixed point set $M^{G}$ is a homological manifold
over $\mathbb{Z}/p$ of codimension even (cf. Lemma \ref{sm} and Lemma \ref%
{Bredon}). Lemma \ref{tmf} implies that $M^{G}$ is a topological manifold.
It is enough to prove the first Betti number of $M^{G}$ is zero. But this is
already known by Edmonds \cite{ad}. For convenience, we repeat the proof.

By Lemma \ref{bred}, $H^{5}(M_{G};\mathbb{Z})\cong H^{5}(\Sigma _{G};\mathbb{%
Z}).$ When $p$ is prime, the action of $G=\mathbb{Z}/p$ is semi-free and $%
\Sigma $ is the fixed point set. Therefore, we have the following
isomorphisms 
\begin{eqnarray*}
H^{5}(M_{G};\mathbb{Z}) &\cong &H^{5}(\Sigma _{G};\mathbb{Z})=H^{5}(\Sigma
\times BG;\mathbb{Z})\cong \tbigoplus\nolimits_{i+j=5}H^{i}(G;H^{j}(F;%
\mathbb{Z)}) \\
&\cong &H^{4}(G;H^{1}(\Sigma ;\mathbb{Z})).
\end{eqnarray*}%
Note that $H^{i}(G;\mathbb{Z})=0$ when $i$ is odd and $\mathbb{Z}/p$ when $%
i>0$ is even. This implies that the Borel spectral sequence $H^{i}(G;H^{j}(M;%
\mathbb{Z}))\Longrightarrow H^{i+j}(M_{G};\mathbb{Z})$ collapses by Lemma %
\ref{collp}. Therefore, the $p$-rank of $H^{5}(M_{G};\mathbb{Z})$ is the
same as that of $\tbigoplus\nolimits_{i+j=5}H^{i}(G;H^{j}(M;\mathbb{Z}%
))\cong 0.$ Thus the first Betti number $b_{1}(\Sigma )=0,$ which gives that 
$\Sigma $ is a disjoint union of spheres $S^{2}$ and discrete points$.$
\end{proof}

\bigskip

We need two more lemmas from Edmonds \cite{ad} (Prop. 1.2 and Cor. 2.6).

\begin{lemma}
\label{euler}Let $g$ be a periodic map of prime order $p$ which acts
orientation-preservingly on a closed four-manifold $M$ with $H_{1}(M;\mathbb{%
Z})=0$. Then the Euler characteristic $\chi (\mathrm{Fix}(g))=\chi
(M)=2+b_{2}(M).$
\end{lemma}

\begin{lemma}
\label{spherehom}Let $g$ be a periodic map of prime order $p$ which acts
orientation-preservingly on a closed four-manifold $M$ with $H_{1}(M;\mathbb{%
Z})=0$. If $\mathrm{Fix}(g)$ is not purely 2-dimensional; then the
2-dimensional components of $\mathrm{Fix}(g)$ represent independent elements
of $H_{2}(M;F_{p})$. If it is purely two-dimensional; and has $k$
two-dimensional components, then the two-dimensional components span a
subspace of $H_{2}(M;F_{p})$ of dimension at least $k-1$; with any $k-1$
components representing independent elements.
\end{lemma}

\begin{corollary}
\label{cor1}Let $M$ be a closed four-manifold $M$ with $H_{1}(M;\mathbb{Z}%
)=0 $ and $G$ a finite group acting on $M$ effectively and homologically
trivially. Suppose that $s\in G$ is of order $p$ and $t$ is a normalizer of
the cyclic subgroup $\langle s\rangle .$ When $b_{2}(M)\geq 3,$ the element $%
t$ acts invariantly and homologically trivially on each 2-sphere in the
fixed point set $\mathrm{Fix}(s)$.
\end{corollary}

\begin{proof}
When $b_{2}(M)\geq 3,$ the Euler characteristic $\chi (M^{s})=2+b_{2}(M)\geq
5$ by Lemma \ref{euler}. Any two spheres in $\mathrm{Fix}(s)$ represent
linearly independent elements of $H_{2}(M;F_{p})$ by Lemma \ref{spherehom}.
Therefore, the element $t$ acts trivially on the homology classes and
preserves each sphere component.
\end{proof}

Corollary \ref{cor1} implies that when $b_{2}(M)\geq 3,$ the fixed point set
of an elementary abelian $p$-group acting effectively and homologically
trivially on a simply connected 4-manifold $M$ is a union of discrete points
and pairwise disjoint 2-spheres.

\section{Finite group actions on 4-manifolds\label{sect}}

In this section, we will prove Theorem \ref{main}. We will follow the
strategy of McCooey \cite{mc}. However, there are several places where
McCooey's argument does not directly generalize. When the group action of $G$
on $M$ is locally linear, the group $G$ acts linearly on the `tangent space' 
$T_{p}M$ for any fixed point $p\in \mathrm{Fix}(G).$ This fact is used
repeatedly in the proof of Proposition 4, the applications of Lemma 10 to
restrict the types of fixed points, the local representation on page 845 of 
\cite{mc} and so on. Without the local linearity, the singular set of the
group action could be very complicated.

We will prove that no non-abelian finite group can act effectively,
homologically trivially on a simply connected closed 4-manifold when the
second Betti number is bigger than $2$. The proof is by contradiction.
Suppose that there is such an action. There would be an action of a minimal
non-abelian group $D$ (i.e. any proper subgroup of $D$ is abelian). Let $%
\Sigma $ be the singular set. We will prove that this is impossible by
computing the cohomology groups $H^{\ast }(M_{D})$ and $H^{\ast }(\Sigma
_{D}),$ which are isomorphic by Lemma \ref{bred} in high dimensions.

First, let us recall the knowledge of minimal non-abelian finite groups. A
non-abelian finite group is minimal if any proper subgroup is abelian. The
minimal non-abelian finite groups are classified as follows according to the
rank of their elementary abelian subgroups ($p,q$ are prime numbers, cf. 
\cite{re,ya}, Section 4 of \cite{mc}):

\begin{itemize}
\item In rank 1, we have the groups 
\begin{equation*}
\mathbb{Z}/p\rtimes \mathbb{Z}/q^{n},
\end{equation*}%
(where the semi-direct product automorphism has order $q)$ and the
quaternion group 
\begin{equation*}
D_{2}^{\ast }=\langle a,b\mid a^{4}=1,a^{2}=b^{2},[a,b]=a^{2}\rangle ;
\end{equation*}

\item In rank 2, we have the groups%
\begin{equation*}
(\mathbb{Z}/p\times \mathbb{Z}/p)\rtimes \mathbb{Z}/q^{n},
\end{equation*}%
(where the semi-direct product automorphism $\sigma $ has order $q$ and $%
\sigma $ acts trivially on any proper invariant submodule of $\mathbb{Z}%
/p\times \mathbb{Z}/p$), 
\begin{eqnarray*}
&&\mathbb{Z}/p^{m}\rtimes \mathbb{Z}/p^{n} \\
&=&\langle a,b\mid a^{p^{m}}=b^{p^{n}}=1,[a,b]=a^{p^{m-1}},m\geq 2\rangle
\end{eqnarray*}%
and $(\mathbb{Z}/p\times \mathbb{Z}/p)\rtimes \mathbb{Z}/p.$ In this case,
there is a normal rank-two subgroup.

\item In rank 3 and higher, we will not need the structure.
\end{itemize}

The case of rank-1 groups will be proved in Theorem \ref{main1} (for $%
\mathbb{Z}/p\rtimes \mathbb{Z}/q^{n},$ $n=1$), Theorem \ref{main2} (for $%
\mathbb{Z}/p\rtimes \mathbb{Z}/q^{n},$ $n>1$) and Theorem \ref{main3} (for
the quaternion group $D_{2}^{\ast }$). The case of high-rank groups will be
proved in Theorem \ref{main4}.

In lots of cases, we will need the following information on cohomology
groups.

\begin{lemma}
\label{cohom} The cohomology groups%
\begin{eqnarray*}
H^{i}(\mathbb{Z}/p\rtimes \mathbb{Z}/q^{n};\mathbb{Z}) &\cong &\left\{ 
\begin{array}{c}
\mathbb{Z},\text{if }i=0 \\ 
\mathbb{Z}/q^{n},\text{ if }\mathit{i}\text{ is even but }2\mathit{%
q\nshortmid i} \\ 
\mathbb{Z}/p\bigoplus \mathbb{Z}/q^{n},\text{ if }2q|i \\ 
0,\text{otherwise.}%
\end{array}%
\right. \\
H^{i}(D_{2}^{\ast };\mathbb{Z}) &\cong &\left\{ 
\begin{array}{c}
\mathbb{Z},\text{if }i=0 \\ 
\mathbb{Z}/2\bigoplus \mathbb{Z}/2,\text{ if }i\equiv 2\func{mod}4 \\ 
\mathbb{Z}/8,\text{ if }i\equiv 0\func{mod}4\text{, }i>0 \\ 
0,\text{otherwise.}%
\end{array}%
\right.
\end{eqnarray*}
\end{lemma}

\subsection{\noindent \textbf{Singular sets}}

In this subsection, we will study the singular set $\Sigma $ for the action
of the minimal non-abelian group $D=\mathbb{Z}/p\rtimes \mathbb{Z}/q^{n}$ $%
(n\geq 1)$ on a 4-manifold $M.$ Let $\mathbb{Z}/p=\langle s\rangle $ be a
cyclic group of prime order $p$ and $\mathbb{Z}/q^{n}=\langle t\rangle $ a
cyclic group of prime order $q^{n}$ acting on $\mathbb{Z}/p.$ Note that $%
q\mid p-1$ and $tst^{-1}=s^{k}$ for some $1<k<p.$ For example, when $%
q=2,n=1, $ we have that 
\begin{equation*}
D=D_{p}=\mathbb{Z}/p\rtimes \mathbb{Z}/2=\langle
s,t:s^{p}=t^{2}=1,tst^{-1}=s^{-1}\rangle
\end{equation*}%
is the dihedral group.

From the definition that $\Sigma =\{x\in M\mid $the stabilizer $D_{x}\neq
\{e\}\},$ we know that $\Sigma =\cup _{g\in G\backslash \{e\}}\mathrm{Fix}%
(g).$

We consider the case when $n=1$ first. For any integer $i,$ we have $%
ts^{i}t^{-1}=s^{j}$ for some integer $j.$ For any integers $k,0<l<q$, we
have $\langle s^{k}t^{l}\rangle =\langle s^{k^{\prime }}t\rangle $ for some $%
k^{\prime }.$ Therefore, the singular set $\Sigma =\cup \mathrm{Fix}(s)\cup
_{i=0}^{p-1}\mathrm{Fix}(s^{i}t).$ The following lemma gives the global
picture of the singular set.

\begin{lemma}
\label{fix}Let the group $D=\mathbb{Z}/p\rtimes \mathbb{Z}/q$ act
effectively homologically trivially on a closed 4-manifold $M$ with $H_{1}(M;%
\mathbb{Z})=0$ by homeomorphisms. The singular set $\Sigma =\cup \mathrm{Fix}%
(s)\cup _{i=0}^{p-1}\mathrm{Fix}(s^{i}t)$ satisfies the following:

\begin{enumerate}
\item[(1)] each fixed point set $\mathrm{Fix}(s^{i}t)$ or $\mathrm{Fix}(s)$
is a disjoint union of 2-spheres and discrete points. Moreover, the Euler
characteristic $\chi (M)=\chi (\mathrm{Fix}(s))=\chi (\mathrm{Fix}(s^{i}t))$
for each $i=0,...,p-1;$

\item[(2)] the intersection $\mathrm{Fix}(s)\cap \mathrm{Fix}(s^{i}t)$ (or $%
\mathrm{Fix}(s^{i}t)\cap \mathrm{Fix}(s^{j}t),i\neq j$) is a part of the
global fixed point set $\mathrm{Fix}(D)=M^{D};$

\item[(3)] the global fixed point set $M^{D}$ is a disjoint union of
(possibly empty) discrete points, 1-spheres and 2-spheres.
\end{enumerate}
\end{lemma}

\begin{proof}
Since each element of $\{s^{i}t\mid i=0,\ldots ,p-1\}$ is of prime order $q$
and the group action is orientation-preserving, it follows Lemma \ref{first}
that the fixed point set is a disjoint union of 2-spheres and discrete
points. The equalities of Euler characteristics are classical (cf. Lemma \ref%
{euler}).

Since any two distinct elements in $\{s,s^{i}t\mid i=0,\ldots ,p-1\}$
generate the whole group $D,$ the claim (2) is proved. Each component of $%
M^{D}$ comes from an action of $t$ on $\mathrm{Fix}(s),$ which is discrete
or an $r$-sphere $(r=0,1,$ or $2).$ This proves (3).
\end{proof}

From (1) of the previous lemma, a connected component in $\mathrm{Fix}(t)$
or $\mathrm{Fix}(s)$ is either a discrete point or a $2$-sphere. A connected
component in the intersection $\mathrm{Fix}(t)\cap \mathrm{Fix}(s)$ can be
either a discrete point or a whole 2-sphere in the singular set $\Sigma .$
The following result implies that two distinct 2-spheres in $\mathrm{Fix}%
(t)\cup \mathrm{Fix}(s)$ cannot have a non-trivial intersection.

\begin{theorem}
\label{n=1}Let the group $D=\mathbb{Z}/p\rtimes \mathbb{Z}/q$ act
effectively homologically trivially on a closed 4-manifold $M$ with $H_{1}(M;%
\mathbb{Z})=0$ by homeomorphisms. Suppose that a discrete point $x\in 
\mathrm{Fix}(D)$ is contained in a non-discrete component of $\mathrm{Fix}(%
\mathbb{Z}/q)$. Then $x$ is a discrete point in $\mathrm{Fix}(\mathbb{Z}/p).$
\end{theorem}

\begin{proof}
It's already known that $\mathrm{Fix}(\mathbb{Z}/p)$ is a union of $2$%
-spheres and discrete points. Suppose that $x\in S^{2}\subset \mathrm{Fix}(%
\mathbb{Z}/q).$ Suppose that $tst^{-1}=s^{k}$ for some $1<k<p.$ Then $%
s^{i}S^{2}=Fix(s^{i}ts^{-i})=Fix(s^{i-k}t)$ for each $1\leq i\leq p-1.$ When 
$x$ is discrete in $\mathrm{Fix}(D),$ we know that $\cap
_{i=1}^{p-1}s^{i}S^{2}=\{x\}.$ Let $U$ be a $D$-invariant open small ball
containing $x$ and $\Sigma ^{\prime }=\Sigma \cap U.$ Then we have $%
H_{c}^{k}(U_{D})=H_{c}^{k}(\Sigma _{G}^{\prime })$ when $k>4$ by Lemma \ref%
{bred} (where $H_{c}^{k}(-)$ is the cohomology group with compact supports).
Note that $H_{c}^{k}(U)=\mathbb{Z}$ when $k=4$ and $0,$ when $k\neq 4.$
Moreover, $H_{c}^{k}(\Sigma ^{\prime })=\lim\limits_{\rightarrow
}H^{k}(\Sigma ^{\prime },\Sigma ^{\prime }-K)$ (the direct limit of the
relative singular cohomology group with $K$ ranges over compact subsets).

Suppose that $x$ belongs to a $2$-sphere $S_{p}^{2}$ component of $\mathrm{%
Fix}(\mathbb{Z}/p).$ Then $\Sigma ^{\prime }$ is a union of $p+1$ 2-discs
with a common intersection point $x.$ Note that $\mathbb{Z}/q$ acts
invariantly on the 2-sphere $S_{p}^{2}.$ Therefore, $H_{c}^{2}(\Sigma
^{\prime })=\mathbb{Z}^{p+1}$ and $H_{c}^{1}(\Sigma ^{\prime })=\mathbb{Z}%
^{p}$ (viewed as the $D$-module $\mathbb{Z}[D]\bigotimes_{\mathbb{Z}[\mathbb{%
Z}/q]}\mathbb{Z}$) and $H_{c}^{0}(\Sigma ^{\prime })=0.$

When $q=2,$ we have the following 
\begin{eqnarray*}
H^{6}(D;H_{c}^{2}(\Sigma ^{\prime })) &\cong &H^{6}(D;\mathbb{Z)}\bigoplus
H^{6}(\mathbb{Z}/q;\mathbb{Z)} \\
&\cong &\mathbb{Z}/q\bigoplus \mathbb{Z}/q,
\end{eqnarray*}%
\begin{eqnarray*}
H^{7}(D;H_{c}^{1}(\Sigma ^{\prime })) &\cong &H^{7}(D;\mathbb{Z}%
[D]\bigotimes\nolimits_{\mathbb{Z}[\mathbb{Z}/q]}\mathbb{Z}) \\
&\cong &H^{7}(\mathbb{Z}/q;\mathbb{Z})=0.
\end{eqnarray*}%
Therefore, the Borel spectral sequence $H^{i}(D;H_{c}^{j}(\Sigma ^{\prime
}))\Longrightarrow H_{c}^{i+j}(\Sigma _{G}^{\prime })$ (cf. \cite{br},
Theorem 9.5, page 251) implies that $H^{8}(\Sigma _{G}^{\prime })$ is
isomorphic to a quotient group of $H^{6}(D;H_{c}^{2}(\Sigma ^{\prime
}))\cong \mathbb{Z}/q\bigoplus \mathbb{Z}/q.$ However, 
\begin{equation*}
H^{8}(U_{G})\cong H^{4}(D;H_{c}^{4}(U))\cong \mathbb{Z}/p\bigoplus \mathbb{Z}%
/2
\end{equation*}
when $q=2$, which is not isomorphic to $H^{8}(\Sigma _{G}^{\prime })$
considering the $p$-part.

When $q>2,$ we have 
\begin{eqnarray*}
H^{2q+2}(D;H_{c}^{2}(\Sigma ^{\prime })) &\cong &H^{2q+2}(D;\mathbb{Z)}%
\bigoplus H^{2q+2}(\mathbb{Z}/q;\mathbb{Z)} \\
&\cong &\mathbb{Z}/q\bigoplus \mathbb{Z}/q
\end{eqnarray*}%
and the Borel spectral sequence gives that $H_{c}^{2q+4}(\Sigma _{G}^{\prime
})$ is isomorphic to a quotient group of $\mathbb{Z}/q\bigoplus \mathbb{Z}%
/q. $ However, 
\begin{equation*}
H_{c}^{2q+4}(U_{G})\cong H^{2q}(D;H_{c}^{4}(U))\cong \mathbb{Z}/p\bigoplus 
\mathbb{Z}/q,
\end{equation*}%
which is not isomorphic to $H_{c}^{2q+4}(\Sigma _{G}^{\prime })$ considering
the $p$-part. This proves that $x$ is a discrete component in $\mathrm{Fix}(%
\mathbb{Z}/p).$
\end{proof}

We now consider the case when $n>1.$ Since the $p$-th power $%
(s^{i}t^{q^{j}})^{p}=t^{pq^{j}}$ for each $i$ and $j>0$ and $%
(s^{i}t)^{q}=t^{q},$ we see that the fixed point set $\mathrm{Fix}%
(s^{i}t^{q^{j}})\subset \mathrm{Fix}(t^{q^{n-1}}),$ which is also a disjoint
union of discrete points and 2-dimensional spheres. Therefore, the singular $%
\Sigma =\mathrm{Fix}(s)\cup \mathrm{Fix}(t^{q^{n-1}})$ is a union of spheres
and discrete points.

\begin{theorem}
\label{n>2}Let the group $D=\mathbb{Z}/p\rtimes \mathbb{Z}/q^{n}$ $(n>1)$
act effectively homologically trivially on a closed 4-manifold $M$ with $%
H_{1}(M;\mathbb{Z})=0$ by homeomorphisms. Suppose that a discrete point $%
x\in \mathrm{Fix}(D)$ is contained in a non-discrete component of $\mathrm{%
Fix}(t^{q^{n-1}})$, then $x$ is a discrete point in $\mathrm{Fix}(\mathbb{Z}%
/p).$
\end{theorem}

\begin{proof}
Suppose that $x\in S_{1}^{2}\cap S_{2}^{2}$ for a 2-sphere $S_{1}^{2}\subset 
\mathrm{Fix}(t^{q^{n-1}})$ and another 2-sphere $S_{2}^{2}\subset \mathrm{Fix%
}(\mathbb{Z}/p).$ Since $s$ commutes with $t^{q^{n-1}},$ we have $%
sS_{2}^{2}=S_{2}^{2}$ as a set. The quotient group $D/\langle
t^{q^{n-1}}\rangle \cong \mathbb{Z}/p\rtimes \mathbb{Z}/q$ acts invariantly
on $\mathrm{Fix}(t^{q^{n-1}}).$ Let $U$ be an $D$-invariant open small ball
containing $x$ and $\Sigma ^{\prime }=\Sigma \cap U.$ Note that $%
H_{c}^{4}(U)=\mathbb{Z}$, $H_{c}^{2}(\Sigma ^{\prime })=\mathbb{Z}^{2}$ and $%
H_{c}^{1}(\Sigma ^{\prime })=\mathbb{Z},H_{c}^{0}(\Sigma ^{\prime })=0$,
with trivial $D$-actions.

When $q=2$, we have 
\begin{eqnarray*}
H^{8}(U_{G}) &\cong &H^{4}(D;H_{c}^{4}(U))\cong \mathbb{Z}/p\tbigoplus 
\mathbb{Z}/2, \\
H^{6}(D;H_{c}^{2}(\Sigma ^{\prime })) &\cong &H^{6}(D;\mathbb{Z}^{2})\cong (%
\mathbb{Z}/q^{n})^{2}, \\
H^{7}(D;H_{c}^{1}(\Sigma ^{\prime })) &=&0.
\end{eqnarray*}%
The Borel spectral sequence $H^{i}(D;H_{c}^{j}(\Sigma ^{\prime
}))\Longrightarrow H_{c}^{i+j}(\Sigma _{G}^{\prime })$ implies that $%
H^{8}(\Sigma _{G}^{\prime })$ is isomorphic to a quotient group of $(\mathbb{%
Z}/q^{n})^{2},$ which has a trivial $p$-part. This is a contradiction to
that $H^{8}(\Sigma _{G}^{\prime })\cong H^{8}(U_{G})$.

When $q>2,$ we have 
\begin{eqnarray*}
H^{2q+4}(U_{G}) &\cong &H^{2q}(D;H_{c}^{4}(U))\cong \mathbb{Z}/p\tbigoplus 
\mathbb{Z}/q^{n}, \\
H^{2q+2}(D;H_{c}^{2}(\Sigma ^{\prime })) &\cong &H^{2q+2}(D;\mathbb{Z}%
^{2})\cong (\mathbb{Z}/q^{n})^{2}, \\
H^{2q+3}(D;H_{c}^{1}(\Sigma ^{\prime })) &=&0.
\end{eqnarray*}%
The Borel spectral sequence again implies that $H^{2q+4}(\Sigma _{G}^{\prime
})$ is isomorphic to a quotient of $(\mathbb{Z}/q^{n})^{2}$ and thus not
isomorphic to $H^{2q+4}(U_{G}),$ considering the $p$-part. This is a
contradiction.
\end{proof}

\subsection{Group actions of $\mathbb{Z}/p\rtimes \mathbb{Z}/q$}

In this subsection, we will prove the following.

\begin{theorem}
\label{main1}Let $p,q$ be two primes and $D=\mathbb{Z}/p\rtimes \mathbb{Z}/q$
a nontrivial semi-direct product. Suppose that $M$ is a 4-dimensional
manifold with $H_{1}(M;\mathbb{Z})=0$ and the second Betti number $%
b_{2}(M)\geq 3.$ The group $D$ cannot act effectively homologically
trivially on $M$ by homeomorphisms.
\end{theorem}

The proof of this theorem is based on a detailed study of singular sets and
applications of Borel spectral sequences.

\begin{lemma}
\label{twosphere}Two distinct 2-spheres $S_{1}^{2},S_{2}^{2}\subset \mathrm{%
Fix}(t)$ have disjoint orbits under the action of $\mathbb{Z}/p=\langle
s\rangle .$
\end{lemma}

\begin{proof}
Suppose that $s^{i}S_{1}^{2}\cap s^{j}S_{2}^{2}\neq \emptyset $ for some $%
0\leq i\neq j<p.$ Then $S_{1}^{2}\cap s^{j-i}S_{2}^{2}\neq \emptyset $ and
there is a point $x\in S_{1}^{2}\cap s^{j-i}S_{2}^{2}.$ Note that $tx=x$ and 
$s^{j-i}ts^{i-j}x=x.$ Therefore, $x\in \mathrm{Fix}(D),$ a global fixed
point. However, $x=s^{i-j}x\in S_{2}^{2}$ implies $S_{1}^{2}\cap
S_{2}^{2}\neq \emptyset ,$ which is a contradiction to the fact that each
component of $\mathrm{Fix}(t)$ is a manifold (see Lemma \ref{fix}).
\end{proof}

From the previous subsection, we know that $\Sigma =\cup \mathrm{Fix}(s)\cup
_{i=0}^{p-1}\mathrm{Fix}(s^{i}t).$ Moreover, the 2-spheres in $\mathrm{Fix}%
(s)$ and $\mathrm{Fix}(t)$ can only be disjoint or identical by Theorem \ref%
{n=1} (the intersection of the two spheres cannot be a circle, since $t$
acts homologically trivial on $\mathrm{Fix}(s)$). By Lemma \ref{twosphere},
two 2-spheres in $\cup _{i=0}^{p-1}\mathrm{Fix}(s^{i}t)$ from distinct $%
\langle s\rangle $-orbits are disjoint. Therefore, a connected component $X$
of $\Sigma $ is either a isolated point, or a single 2-sphere, or a union of 
$p$ 2-spheres (from $\cup _{i=0}^{p-1}\mathrm{Fix}(s^{i}t)$) having
non-empty intersections. The remaining part of the proof of Theorem \ref%
{main1} will be the similar to that of Proposition 13 in McCooey \cite{mc}.

Suppose that the singular set $\Sigma $ has $k$ isolated global fixed
points, $l$ connected components of single 2-spheres in $\mathrm{Fix}(D)$
and $n_{1}$ connected components of unions of $p$ 2-spheres with non-empty
intersections. Suppose that the rest of $\mathrm{Fix}(s)$ has $n_{2}$ free $%
\langle t\rangle $-orbits of 2-spheres, and $n_{3}$ free $\langle t\rangle $%
-orbits of isolated points. Moreover, the rest of $\mathrm{Fix}(t)$ contains 
$n_{4}$ free $\langle s\rangle $-orbits of $2$-spheres and $n_{5}$ free $%
\langle s\rangle $-orbits of isolated points.

The following was first essentially obtained by McCooey \cite{mc} (page 847).

\begin{lemma}
\label{mccolap}Let $\Sigma ^{\prime }\subset \Sigma $ be the union of the
connected components which are not isolated points or single 2-spheres in $%
\mathrm{Fix}(D).$ Suppose that $\Sigma ^{\prime }$ has $n_{1}$ connected
components consisting of unions of $p$ 2-spheres intersecting at $m_{i}\geq
1 $ points. The Borel spectral sequence for $H^{\ast }(\Sigma _{D}^{\prime
};Z) $ collapses. Actually, we have%
\begin{eqnarray*}
H^{n}(D;H^{0}(\Sigma ^{\prime })) &\cong &\left\{ 
\begin{array}{c}
0,\text{if }n\text{ is odd,} \\ 
(\mathbb{Z}/q)^{n_{1}+n_{4}+n_{5}}\bigoplus (\mathbb{Z}/p)^{n_{2}+n_{3}},%
\text{ if }\mathit{i}\text{ is even but }2\mathit{q\nshortmid n} \\ 
(\mathbb{Z}/q)^{n_{1}+n_{4}+n_{5}}\bigoplus (\mathbb{Z}%
/p)^{n_{1}+n_{2}+n_{3}},\text{ if }2q|n.%
\end{array}%
\right. \\
H^{n}(D;H^{1}(\Sigma ^{\prime })) &\cong &\left\{ 
\begin{array}{c}
\bigoplus_{i=1}^{n_{i}}(\mathbb{Z}/p)^{m_{i}-1},\text{if }n\equiv 2q-1\func{%
mod}2q\text{,} \\ 
0,\text{ otherwise}\mathit{.}%
\end{array}%
\right. \\
H^{n}(D;H^{2}(\Sigma ^{\prime })) &\cong &\left\{ 
\begin{array}{c}
0,\text{if }n\text{ is odd,} \\ 
(\mathbb{Z}/q)^{n_{1}+n_{4}}\bigoplus (\mathbb{Z}/p)^{n_{2}},\text{ if }n%
\text{ is even}\mathit{.}%
\end{array}%
\right.
\end{eqnarray*}
\end{lemma}

\begin{proof}
When $q=2,$ this is proved by McCooey \cite{mc}. For $q>2,$ the proof is
similar. For completeness, we repeat it here. Let $X$ be a union of $p$
2-spheres intersecting at $m_{i}\geq 1$ points $x_{1},x_{2},...,x_{m_{i}}.$
The $D$-equivariant chain complex for $X$ is of the form%
\begin{equation*}
0\rightarrow \mathbb{Z}D\tbigotimes\nolimits_{\langle t\rangle }\mathbb{Z}%
\rightarrow (\mathbb{Z}D\tbigotimes\nolimits_{\langle t\rangle }\mathbb{Z}%
)^{m_{i}-1}\rightarrow \mathbb{Z}^{m_{i}}\rightarrow 0.
\end{equation*}%
Let $Ck(\mathbb{Z})=$coker$(\mathbb{Z\rightarrow }\mathrm{Hom}_{\mathbb{Z}[%
\mathbb{Z}/q]}(\mathbb{Z}[D],\mathbb{Z})),$ i.e.%
\begin{equation*}
0\rightarrow \mathbb{Z}\rightarrow \mathrm{Hom}_{\mathbb{Z}[\mathbb{Z}/q]}(%
\mathbb{Z}[D],\mathbb{Z})\rightarrow Ck(\mathbb{Z})\rightarrow 0.
\end{equation*}%
The Shapiro's Lemma implies that $H^{i}(D;\mathrm{Hom}_{\mathbb{Z}[\mathbb{Z}%
/q]}(\mathbb{Z}[D],\mathbb{Z}))\cong H^{i}(\mathbb{Z}/q;\mathbb{Z})$ and the
long exact sequence gives 
\begin{equation*}
\cdots \rightarrow H^{i}(D;\mathbb{Z})\rightarrow H^{i}(\mathbb{Z}/q;\mathbb{%
Z})\rightarrow H^{i}(D;Ck(\mathbb{Z}))\rightarrow H^{i+1}(D;\mathbb{Z}%
)\rightarrow \cdots .
\end{equation*}%
Therefore, 
\begin{equation*}
H^{n}(D;Ck(\mathbb{Z}))\cong \left\{ 
\begin{array}{c}
\mathbb{Z}/p,\text{if }n\equiv 2q-1\func{mod}2q\text{,} \\ 
0,\text{ otherwise}\mathit{.}%
\end{array}%
\right.
\end{equation*}%
The $D$-equivariant chain complex for $X$ gives 
\begin{eqnarray*}
H^{0}(X;\mathbb{Z}) &=&\mathbb{Z}, \\
H^{1}(X;\mathbb{Z}) &=&(Ck(\mathbb{Z}))^{m_{i}-1}, \\
H^{2}(X;\mathbb{Z}) &=&\mathrm{Hom}_{\mathbb{Z}[\mathbb{Z}/q]}(\mathbb{Z}[D],%
\mathbb{Z}).
\end{eqnarray*}%
The cohomology groups of $\Sigma ^{\prime }$ are%
\begin{eqnarray*}
H^{0}(\Sigma ^{\prime };\mathbb{Z}) &\cong &\mathbb{Z}^{n_{1}}\tbigoplus 
\mathrm{Hom}_{\mathbb{Z}[\mathbb{Z}/p]}(\mathbb{Z}[D],\mathbb{Z}%
)^{n_{2}+n_{3}}\tbigoplus \mathrm{Hom}_{\mathbb{Z}[\mathbb{Z}/q]}(\mathbb{Z}%
[D],\mathbb{Z})^{n_{4}+n_{5}}, \\
H^{1}(\Sigma ^{\prime };\mathbb{Z}) &\cong &\tbigoplus_{i=1}^{n}(Ck(\mathbb{Z%
}))^{m_{i}-1}, \\
H^{2}(\Sigma ^{\prime };\mathbb{Z}) &\cong &\mathrm{Hom}_{\mathbb{Z}[\mathbb{%
Z}/p]}(\mathbb{Z}[D],\mathbb{Z})^{n_{2}}\tbigoplus \mathrm{Hom}_{\mathbb{Z}[%
\mathbb{Z}/q]}(\mathbb{Z}[D],\mathbb{Z})^{n_{1}+n_{4}}.
\end{eqnarray*}%
The Shapiro's lemma gives the required isomorphisms of the lemma. Since the $%
E_{2}^{i,2}=H^{i}(D;H^{2}(\Sigma ^{\prime };\mathbb{Z}))$ and $E_{2}^{i+2,0}$
vanish in odd dimensions and $E_{2}^{i,1}$ vanishes in even dimensions, the
Borel spectral sequence collapses.
\end{proof}

\bigskip

\begin{proof}[Proof of Theorem \protect\ref{main1}]
Suppose that $D$ acts effectively homologically trivially on $M.$ Since the
odd-dimensional cohomology groups of $D$ vanish (cf. Lemma \ref{cohom}), the
Borel spectral sequence collapses by Lemma \ref{collp}. Therefore, the
graded module (cf. \cite{tom}, Proposition 1.18 of Chapter III, page 182) 
\begin{equation}
\mathrm{GR}(H^{i}(M_{D}))\cong H^{i}(D;\mathbb{Z})\tbigoplus
H^{i-2}(D;H^{2}(M;\mathbb{Z}))\tbigoplus H^{i-4}(D;\mathbb{Z}).  \tag{*}
\end{equation}%
which is isomorphic to $\mathrm{GR}(H^{i}(\Sigma _{D}))$ when $i>4$ by Lemma %
\ref{Bredon}. We will prove that this is impossible by comparing $%
H^{4q}(M_{D})$ and $H^{4q}(\Sigma _{D}).$

Note that the singular set $\Sigma =\Sigma ^{1}\cup \Sigma ^{2}\cup \Sigma
^{\prime },$ is a disjoint union of $\Sigma ^{1}$ consisting of $k$ isolated
global fixed points, $\Sigma ^{2}$ consisting of $l$ components of single
2-spheres in $\mathrm{Fix}(D),$ and $\Sigma ^{\prime }$ the remaining part.
Therefore,%
\begin{equation*}
H^{n}(\Sigma _{D};\mathbb{Z})\cong H^{n}(\Sigma _{D}^{\prime };\mathbb{Z}%
)\bigoplus H^{n}(\Sigma _{D}^{1};\mathbb{Z})\bigoplus H^{n}(\Sigma _{D}^{2};%
\mathbb{Z}).
\end{equation*}%
Lemma \ref{collp} implies that the Borel spectral sequences for $\Sigma
_{D}^{1}\rightarrow BD$ and $\Sigma _{D}^{2}\rightarrow BD$ collapse. By
Lemma \ref{mccolap}, we have 
\begin{eqnarray*}
&&H^{4q}(G;H^{0}(\Sigma ;\mathbb{Z)})\tbigoplus H^{4q-1}(G;H^{1}(\Sigma ;%
\mathbb{Z)})\tbigoplus H^{4q-2}(G;H^{2}(\Sigma ;\mathbb{Z)}) \\
&\cong &(\mathbb{Z}/pq)^{k+l+n_{1}}\tbigoplus ((\mathbb{Z}%
/p)^{n_{2}+n_{3}}\tbigoplus (\mathbb{Z}/q)^{n_{4}+n_{5}})\tbigoplus (\mathbb{%
Z}/p)^{\tsum (m_{i}-1)} \\
&&\tbigoplus (\mathbb{Z}/q)^{n_{1}+n_{4}}\tbigoplus (\mathbb{Z}%
/p)^{n_{2}}\tbigoplus (\mathbb{Z}/q)^{l}.
\end{eqnarray*}%
Since 
\begin{equation*}
H^{4q}(M_{G})\cong \mathbb{Z}/pq\tbigoplus (\mathbb{Z}/q)^{b_{2}(M)}%
\tbigoplus \mathbb{Z}/pq
\end{equation*}%
when $q=2$ (or $\mathbb{Z}/pq\tbigoplus (\mathbb{Z}/q)^{b_{2}(M)}$ when $q>2$%
), we have that 
\begin{equation*}
k+l+n_{1}+2n_{2}+n_{3}+\tsum (m_{i}-1)=2
\end{equation*}%
by considering the $p$-part. When $n_{1}=0,$ we have that $\Sigma ^{\prime }$
is empty and $k+l+2n_{2}+n_{3}=2$. Since $2+b_{2}(M)=\chi
(Fix(s))=k+2l+4n_{2}+2n_{3}\leq 4,$ we get that $b_{2}(M)\leq 2.$ When $%
n_{1}=1,$ we have that $n_{2}=0$ and $k+l+n_{3}+(m_{1}-1)\leq 1.$ Once
again, $2+b_{2}(M)=\chi (Fix(s))=k+2l+2n_{3}+m_{1}\leq 4$ and thus $%
b_{2}(M)\leq 2.$ When $n_{1}=2,$ we have $k+l+2n_{2}+n_{3}+\tsum (m_{i}-1)=0$
and thus $2+b_{2}(M)=\chi (Fix(s))\leq 4,$ implying $b_{2}(M)\leq 2.$ These
are contradictions to $b_{2}(M)\geq 3.$

When $q>2,$ we have that%
\begin{equation*}
k+l+n_{1}+2n_{2}+n_{3}+\tsum (m_{i}-1)=1
\end{equation*}%
by considering the $p$-part. Similar argument shows that $b_{2}(M)\leq 2,$ a
contradiction.
\end{proof}

\bigskip

\subsection{Group actions of $\mathbb{Z}/p\rtimes \mathbb{Z}/q^{n},n>1$}

In this subsection, we study the group action of $\mathbb{Z}/p\rtimes 
\mathbb{Z}/q^{n},n>1$ and prove the following result. We still assume $s$ is
a generator of $\mathbb{Z}/p$ and $t$ is a generator of $\mathbb{Z}/q^{n}.$

\begin{theorem}
\label{main2}Let $p,q$ be two primes and $D=\mathbb{Z}/p\rtimes \mathbb{Z}%
/q^{n}$ $(n>1)$ a minimal non-abelian finite group. Suppose that $M$ is a
four-dimensional closed manifold with $H_{1}(M;\mathbb{Z})=0$ and the the
second Betti number $b_{2}(M)\geq 3.$ The group $D$ cannot act effectively
homologically trivially on $M$ by homeomorphisms.
\end{theorem}

\begin{proof}
Suppose that $D$ acts effectively homologically trivially on $M.$ The
singular set $\Sigma $ is described as following. Note that the fixed point
set $\mathrm{Fix}(s)$ is a disjoint union of discrete points and
2-dimensional spheres. Since the $p$-th power $%
(s^{i}t^{q^{j}})^{p}=t^{pq^{j}}$ for each $i$ and $j>0$ and $%
(s^{i}t)^{q}=t^{q},$ we see that the fixed point set $\mathrm{Fix}%
(s^{i}t^{q^{j}})\subset \mathrm{Fix}(t^{q^{n-1}}),$ which is also a disjoint
union of discrete points and 2-dimensional spheres. Therefore, the singular
set $\Sigma =\mathrm{Fix}(s)\cup \mathrm{Fix}(t^{q^{n-1}})$ is a union of
spheres and discrete points. By Theorem \ref{n>2}, distinct 2-spheres in $%
\mathrm{Fix}(s)\cup \mathrm{Fix}(t^{q^{n-1}})$ cannot have a discrete
intersection. Therefore, the singular set $\Sigma $ consists of (say $n_{1}$%
) isolated fixed points and (say $n_{2}$) 2-spheres. When $b_{2}(M)\geq 3,$
the fixed point set $\mathrm{Fix}(g)$ of any order-$p$ element $g\in D$
cannot be purely 2-dimensional consisting of two 2-spheres (otherwise, the $%
\chi (\mathrm{Fix}(g))=4=2+b_{2}(M),$ a contradiction). Lemma \ref{spherehom}
implies arbitrary two 2-spheres in $\mathrm{Fix}(g)$ are independent in $%
H_{2}(M;F_{p}).$ This implies that any element $h\in N_{D}(\langle g\rangle
) $ (a normalizer of $\langle g\rangle $) preserves each 2-sphere component
of $\mathrm{Fix}(g)$ (see Corollary \ref{cor1}). Therefore, each point and
each sphere in $\Sigma $ is invariant under the action of $D.$

Since the odd-dimensional cohomology groups of $D$ are vanishing (cf. Lemma %
\ref{cohom}), the Borel spectral sequence collapses for both $%
M_{D}\rightarrow BD$ and $\Sigma _{D}\rightarrow BD$ by Lemma \ref{collp}.
We have that%
\begin{eqnarray*}
H^{4q}(\Sigma _{D};\mathbb{Z}) &\cong &H^{4q}(D;H^{0}(\Sigma ;\mathbb{Z}%
))\tbigoplus H^{4q-2}(D;H^{2}(\Sigma ;\mathbb{Z})) \\
&\cong &(\mathbb{Z}/pq^{n})^{n_{1}+n_{2}}\tbigoplus (\mathbb{Z}%
/q^{n})^{n_{2}}, \\
H^{4q}(M_{D};\mathbb{Z}) &\cong &H^{4q}(D;H^{0}(M;\mathbb{Z}))\tbigoplus
H^{4q-2}(D;H^{2}(M;\mathbb{Z}))\tbigoplus H^{4q-4}(D;H^{4}(M;\mathbb{Z})) \\
&\cong &(\mathbb{Z}/pq^{n})^{2}\tbigoplus (\mathbb{Z}/q^{n})^{b_{2}(M)}
\end{eqnarray*}%
when $q=2$ (or $(\mathbb{Z}/pq^{n})\tbigoplus (\mathbb{Z}/q^{n})^{b_{2}+1}$
when $q>2$). By Lemma \ref{Bredon}, $H^{4q}(\Sigma _{D};\mathbb{Z})\cong
H^{4q}(M_{D};\mathbb{Z})$, which gives that $n_{1}+n_{2}=2$ (or $%
n_{1}+n_{2}=1$ when $q>2$) and $n_{1}+2n_{2}=2+b_{2}\leq 4,$ which is a
contradiction to that $b_{2}(M)\geq 3.$
\end{proof}

\subsection{Group actions of the quaternion group}

In this subsection, we study the group action of the quaternion group $%
D_{2}^{\ast }=\langle a,b\mid a^{4}=1,a^{2}=b^{2},[a,b]=a^{2}\rangle .$ We
prove the following.

\begin{theorem}
\label{main3}Suppose that $M$ is a four-dimensional closed manifold with $%
H_{1}(M;\mathbb{Z})=0$ and the the second Betti number $b_{2}(M)\geq 3.$ The
group $D_{2}^{\ast }$ cannot act effectively homologically trivially on $M$
by homeomorphisms.
\end{theorem}

\begin{proof}
Since the order of $a^{2}$ is $2,$ the fixed point set $\mathrm{Fix}(a^{2})$
is a disjoint union of 2-spheres and discrete points by Lemma \ref{first}.
Note that any non-trivial element $g\neq a^{2}$ has $g^{2}=a^{2}.$ This
implies that $\mathrm{Fix}(g)\subset \mathrm{Fix}(a^{2})$ and the singular
set $\Sigma =\mathrm{Fix}(a^{2}).$ Suppose that $\mathrm{Fix}(a^{2})$
consists of $n_{1}$ discrete points and $n_{2}$ 2-spheres. The quotient
group $D_{2}^{\ast }/\langle a^{2}\rangle $ acts invariantly on the fixed
point set $\mathrm{Fix}(a^{2}),$ preserving the $n_{2}$ individual
2-spheres, while the action could permute the $n_{1}$ discrete points. Since
each 2-sphere in $\mathrm{Fix}(a^{2})$ represents a non-trivial homology
class, $D_{2}^{\ast }/\langle a^{2}\rangle $ preserves each such a 2-sphere.
Let $n_{1i}$ $(i=0,1,2)$ denote the number of $D_{2}^{\ast }/\langle
a^{2}\rangle $-orbits of discrete points in $\mathrm{Fix}(a^{2})$ with
stabilizers of $2$-rank $i$. Since the odd-dimensional cohomology groups of $%
G:=D_{2}^{\ast }$ vanish (cf. Lemma \ref{cohom}), the Borel spectral
sequence collapses by Lemma \ref{collp}. Note that 
\begin{equation*}
H^{0}(\Sigma ;\mathbb{Z})\cong (\mathbb{Z[}G]\tbigotimes\nolimits_{\mathbb{%
Z\langle }a^{2}\rangle }\mathbb{Z)}^{n_{10}}\tbigoplus (\mathbb{Z[}%
G]\tbigotimes\nolimits_{\mathbb{Z[Z}/4]}\mathbb{Z)}^{n_{11}}\tbigoplus 
\mathbb{Z}^{n_{12}+n_{2}}.
\end{equation*}%
Therefore, we have the isomorphisms of graded modules%
\begin{eqnarray*}
\mathrm{GR}(H^{8}(M_{G};\mathbb{Z})) &\cong &H^{8}(G;\mathbb{Z})\tbigoplus
H^{6}(G;H^{2}(M;\mathbb{Z}))\tbigoplus H^{4}(G;H^{4}(M;\mathbb{Z)}) \\
&\cong &(\mathbb{Z}/8)^{2}\tbigoplus (\mathbb{Z}/2)^{2b_{2}(M)}, \\
\mathrm{GR}(H^{8}(\Sigma _{G};\mathbb{Z})) &\cong &H^{8}(G;H^{0}(\Sigma ;%
\mathbb{Z)})\tbigoplus H^{6}(G;H^{2}(\Sigma ;\mathbb{Z)}), \\
&\cong &(\mathbb{Z}/8)^{n_{12}+n_{2}}\tbigoplus (\mathbb{Z}%
/4)^{n_{11}}\tbigoplus (\mathbb{Z}/2)^{n_{10}}\tbigoplus (\mathbb{Z}%
/2)^{2n_{2}}
\end{eqnarray*}%
and%
\begin{eqnarray*}
\mathrm{GR}(H^{6}(M_{G};\mathbb{Z})) &\cong &H^{6}(G;\mathbb{Z})\tbigoplus
H^{4}(G;H^{2}(M;\mathbb{Z)})\tbigoplus H^{2}(G;H^{4}(M;\mathbb{Z)}) \\
&\cong &(\mathbb{Z}/8)^{b_{2}(M)}\tbigoplus (\mathbb{Z}/2)^{4}, \\
\mathrm{GR}(H^{6}(\Sigma _{G};\mathbb{Z})) &\cong &H^{6}(G;H^{0}(\Sigma ;%
\mathbb{Z)})\tbigoplus H^{4}(G;H^{2}(\Sigma ;\mathbb{Z)}) \\
&\cong &(\mathbb{Z}/2)^{2(n_{12}+n_{2})}\tbigoplus (\mathbb{Z}%
/4)^{n_{11}}\tbigoplus (\mathbb{Z}/2)^{n_{10}}\tbigoplus (\mathbb{Z}%
/8)^{n_{2}}.
\end{eqnarray*}%
Considering Lemma \ref{Bredon}, the cardinalities of modules give that $%
6+2b_{2}(M)=3(n_{12}+n_{2})+2n_{11}+n_{10}+2n_{2}$ and $%
3b_{2}(M)+4=2(n_{12}+n_{2})+2n_{11}+n_{10}+3n_{2}.$ Thus, $%
b_{2}(M)=2-n_{12}\leq 2,$ a contradiction.
\end{proof}

\subsection{High rank case}

In this subsection, we will prove that a high-rank finite group acting
effectively homologically trivially on a simply connected 4-manifold $M$
with large second Betti number has to be $\mathbb{Z}/p\times \mathbb{Z}/p.$

\begin{lemma}[Borel \protect\cite{Bo}, Theorem 4.3, p.182]
\label{borel}Let $G$ be an elementary $p$-group operating on a first
countable cohomology $n$-manifold $X$ mod $p.$ Let $x\in X$ be a fixed point
of $G$ on $X$ and let $n(H)$ be the cohomology dimension mod $p$ of the
component of $x$ in the fixed point set of a subgroup $H$ of $G.$ If $%
r=n(G), $ we have 
\begin{equation*}
n-r=\sum_{H}(n(H)-r)
\end{equation*}%
where $H\ $runs through the subgroups of $G$ of index $p.$
\end{lemma}

The following is a special case of Lemma 2.1 in \cite{mc2}.

\begin{lemma}
\label{mccol}Let $M$ be a closed 4-manifold with $H_{1}(M;F_{p})=0$ and $G$
a finite group acting homologically trivially on $M.$ When $b_{2}(M)\geq 3,$
the Borel spectral sequence 
\begin{equation*}
H^{i}(G;H^{j}(M;F_{p}))\Longrightarrow H^{i+j}(M_{G};F_{p})
\end{equation*}%
collapses.
\end{lemma}

\begin{lemma}
\label{elem}Let $G=(\mathbb{Z}/p)^{k}$ act effectively homologically
trivially on a 4-manifold $M$ with $H_{1}(M;\mathbb{Z})=0.$

\begin{enumerate}
\item[(i)] If $b_{2}(M)\geq 3,$ we have $k\leq 2.$ When $k=2,$ the singular
set $\Sigma $ consists of chains of 2-spheres arranged in closed loops.
Moreover, the global fixed point $\mathrm{Fix}(G)$ consists of $b_{2}(M)+2$
points.

\item[(ii)] If $b_{2}(M)=1,p=2$ and $k=2,$ the singular set $\Sigma $
consists of chains of $3$ 2-spheres arranged in a closed loop. The global
fixed point $\mathrm{Fix}(G)$ consists of the $3$ intersection points.

\item[(iii)] If $b_{2}(M)=2,p>2$ and $k=2,$ the singular set $\Sigma $
consists of either two loops of two 2-spheres or a single loop of four
2-spheres, linking together at global fixed points.
\end{enumerate}
\end{lemma}

\begin{proof}
When $b_{2}(M)\geq 3$, the Borel spectral sequence collapses (cf. Lemma \ref%
{mccol}). Therefore, the module $H^{\ast }(M_{G};F_{p})$ is isomorphic to $%
H^{\ast }(G;F_{p})\otimes H^{\ast }(M;F_{p}).$ Recall that $H^{\ast
}(G;F_{p})$ is isomorphic to $F_{p}[x_{1},\cdots ,x_{k}]$ (with $\deg
x_{i}=1 $ for each $i$) when $p=2,$ or $F_{p}[y_{1},\cdots ,y_{k}]\bigotimes
\Lambda \lbrack x_{1},\cdots ,x_{k}]$ (with $\deg x_{i}=1$ and $\deg y_{j}=2$
for any $i,j$) when $p>2$ (see \cite{xiang}, p.45). Note that we cannot
apply Lemma \ref{collp} to get the collapse of the Borel spectral sequence,
since the odd dimensional $F_{p}$-coefficient cohomology groups may not
vanish. Let $S$ be the multiplicative set of non-trivial elements in the
polynomial ring $F_{p}[x_{1},\cdots ,x_{k}]$ (or $F_{p}[y_{1},\cdots ,y_{k}]$%
). Then the set 
\begin{equation*}
M^{S}=\{x\in M:S\cap \ker (H^{\ast }(G;F_{p})\rightarrow H^{\ast
}(G_{x};F_{p}))=\emptyset \}
\end{equation*}%
would be the global fixed point set $M^{G}.$ The localization theory gives
an isomorphism 
\begin{equation*}
S^{-1}H_{G}^{\ast }(M;F_{p})\cong S^{-1}H_{G}^{\ast }(M^{G};F_{p})\cong
S^{-1}H^{\ast }(G;F_{p})\otimes H^{\ast }(M^{G};F_{p}).
\end{equation*}%
This clearly implies that $M^{G}$ is non-empty. From Lemma \ref{first}, we
know that the fixed point set $\mathrm{Fix}(H)$ is a union of (possibly
empty) 2-spheres $S^{2}$ and discrete points for any subgroup $H<G$ of index 
$p$. From the Borel's formula (cf. Lemma \ref{borel}) 
\begin{equation*}
4-n_{0}=\sum n(H)-n_{0},
\end{equation*}%
where $n(H)$ is the dimension of $\mathrm{Fix}(H)$ and $n_{0}$ is the
dimension of $\mathrm{Fix}(G),$ we know that there are must be some $H$ such
that $n(H)=2$ and $n_{0}=0.$ We may assume that for some nontrivial element $%
a\in G,$ the fixed point set $\mathrm{Fix}(a)$ contains $S^{2},$ which
contains a global fixed point in $M^{G}.$ Fix a decomposition $G=\langle
a\rangle \times (\mathbb{Z}/p)^{k-1}$. Then the complement $(\mathbb{Z}%
/p)^{k-1}$ acts invariantly on $S^{2},$ since two different $S^{2}$
components in $\mathrm{Fix}(a)$ represent different homology classes in $%
H_{2}M$ (cf. Lemma \ref{cor1}).

Suppose that $k\geq 3.$ When $p>2,$ the subgroup $(\mathbb{Z}/p)^{2}$ cannot
act effectively on $S^{2}$ as seen from the Borel's formula. Thus there is
an element $g\in (\mathbb{Z}/p)^{k-1}$ acting trivially on $S^{2}.$ Then the
subgroup $\langle a,g\rangle $ acts effectively on $M$ with a global fixed
point set $S^{2}.$ However, this is impossible by Borel's formula again.
When $p=2,$ each element in $(\mathbb{Z}/p)^{2}$ acts on $S^{2}$ by
orientation-preserving since the codimension of the fixed point set of each
element is even. Once again, the action of $(\mathbb{Z}/2)^{2}$ cannot be
effective, since there is a global fixed point in $S^{2}$. Using Borel's
formula again, this is impossible.

The only admitted case is that $k=2$ and each generator $a,b$ of $G=\langle
a\rangle \times \langle b\rangle \cong (\mathbb{Z}/p)^{2}$ fixes several
copies of $S^{2}$. Moreover, one generator rotates the spheres fixed by the
other generator. Therefore, the singular set is some chains of 2-spheres $%
S_{1},S_{2},\cdots ,S_{k}$ arranged in closed loops with $S_{i}$ intersects
once with $S_{i-1}$ and $S_{i+1}$ respectively. Furthermore, the 2-spheres $%
S_{2i-1}$ (resp. $S_{2i}$) are fixed by the generator $a$ (resp. $b$). The
fixed point $\mathrm{Fix}(G)$ consists of $b_{2}(M)+2$ points considering
the Euler characteristics.

When $b_{2}(M)=1$ and $p=2,$ the singular set $\Sigma =\mathrm{Fix}(a)\cup 
\mathrm{Fix}(b)\cup \mathrm{Fix}(ab).$ If $\Sigma $ is of dimension $0,$ the
element $b$ acts invariantly on $\mathrm{Fix}(a)$ consisting of $3$ points
(by Lemma \ref{euler}) and thus has a fixed point. But this is impossible by
the Borel's formula. After changing of basis, we may assume that $\mathrm{Fix%
}(a)$ consists of a 2-sphere $S_{a}^{2}$ and a discrete point $x.$ Since $b$
acts trivially on $x,$ there is a sphere $S_{b}^{2}\subset \mathrm{Fix}(b)$
containing $x$ by the Borel formula. Note that $a$ acts invariantly and
non-trivially on $S_{b}^{2},$ the sphere $S_{b}^{2}$ intersects $S_{a}^{2}$
at a point $y.$ The Borel's formula for $y$ shows that $\mathrm{Fix}(ab)$
contains another sphere $S_{ab}^{2}$ connecting $S_{a}^{2}$ and $x.$

We consider the case when $b_{2}(M)=2$ and $p>2.$ The singular set $\Sigma =%
\mathrm{Fix}(a)\cup _{i=0}^{p-1}\mathrm{Fix}(a^{i}b),$ a union of fixed
point sets. Note that $\mathrm{Fix}(a)$ is a union of (possibly empty)
discrete points and 2-spheres. If $\Sigma $ is of dimension $0,$ the element 
$b$ acts invariantly on $\mathrm{Fix}(a)$ consisting of $4$ points and thus
has a fixed point. But this is impossible by the Borel's formula. If $\Sigma 
$ is of dimension two, after changing of basis we may assume that $\mathrm{%
Fix}(a)$ contains a 2-sphere. Note that the Borel's formula implies that at
each global fixed point $z\in \mathrm{Fix}(G),$ there are two elements $%
g,h\in G$ generating different subgroups with two distinct spheres $%
S_{g}^{2}\subset \mathrm{Fix}(g)$ and $S_{h}^{2}\subset \mathrm{Fix}(h)$
such that $z\in S_{g}^{2}\cap S_{h}^{2}$ . Since any element in $G$ acts
invariantly on these spheres, each sphere $S_{g}^{2}$ or $S_{h}^{2}$
contains two global fixed points in $\mathrm{Fix}(G)$. These spheres form
either two loops of 2-spheres or a single loop consisting of four spheres,
linking together at global fixed points.
\end{proof}

\begin{remark}
There are actions of $\mathbb{Z}/2\times \mathbb{Z}/2$ on $S^{2}\times S^{2}$
and actions of $\mathbb{Z}/3\times \mathbb{Z}/3$ on $\mathbb{C}P^{2}$
without global fixed points. The isometry group $(\mathrm{SO}(3)\times 
\mathrm{SO}(3))\rtimes \mathbb{Z}/2$ of $S^{2}\times S^{2}$ and the isometry
group $PU(3)$ of $\mathbb{C}P^{2}$ both contain elementary subgroup of ranks
larger $2.$ This means that the restrictions of Betti numbers and prime $p$
in Lemma \ref{elem} cannot be dropped. More information about group actions
on $\mathbb{C}P^{2}$ can be found in \cite{hl}.
\end{remark}

It is proved in \cite{mc2} (Lemma 2.6) that the singular set $\Sigma $ in
Lemma \ref{elem} is actually connected and there is only one such a loop.
For convenience, we repeat the proof here (note that there are some typos in
the proof of Corollary 2.5 and Lemma 2.6 in \cite{mc2}).

\begin{lemma}
\label{mac}Let $G=(\mathbb{Z}/p)^{2}$ act effectively homologically
trivially on a closed 4-manifold $M$ with $H_{1}(M;\mathbb{Z})=0.$ Suppose
that one the following conditions holds:

(1) $b_{2}(M)\geq 3,$

(2) $b_{2}(M)=1$ and $p=2,$

(3) $b_{2}(M)=2$ and $p>2.$

Then the singular set $\Sigma $ is a chain of 2-spheres arranged in a closed
loop. Moreover, each sphere represents a primitive class in $H_{2}(M;\mathbb{%
Z})$ and together these classes generate $H_{2}(M;\mathbb{Z}).$ The number
of fixed spheres is $b_{2}(M)+2.$
\end{lemma}

\begin{proof}
Let $N$ be the number of connected components of $\Sigma .$ Consider the
long exact sequence%
\begin{eqnarray*}
\cdots &\rightarrow &H^{1}(M,\Sigma )\rightarrow H^{1}(M)\rightarrow
H^{1}(\Sigma ) \\
&\rightarrow &H^{2}(M,\Sigma )\rightarrow H^{2}(M)\rightarrow H^{2}(\Sigma
)\rightarrow H^{3}(M,\Sigma )\rightarrow \cdots
\end{eqnarray*}%
for the pair $(M,\Sigma ),$ we have that $H^{1}(M,\Sigma )\cong \mathbb{Z}%
^{N-1},H^{2}(M,\Sigma )\cong \mathbb{Z}^{N+L},$ where $L$ is the rank of the
cokernel $H^{1}(\Sigma )\rightarrow H^{2}(M,\Sigma ).$ By Lemma \ref{elem},
the singular $\Sigma $ consists of loops of spheres. If a component of $%
\Sigma $ contains at least three spheres, then each sphere intersects its
neighbor geometrically once and thus represents primitive element in $%
H_{2}(M).$ If a component of $\Sigma $ contains only two spheres, the
homology class represented by the sphere may be a multiple of $2.$ But this
case only happen when $p>2,$ since the homology class is non-trivial in $%
H_{2}(M;\mathbb{Z}/2)$ when $p=2$ (cf. Lemma \ref{spherehom}). Therefore,
the relative cohomology $H^{3}(M,\Sigma )$ as the cokernel of $%
H^{2}(M)\rightarrow H^{2}(\Sigma )$ is isomorphic to $\mathbb{Z}%
^{L+2}\bigoplus T$ by counting ranks, where $T=0$ when $p=2$ and $2T=0$ when 
$p>2.$

Let $\pi :M-\Sigma \rightarrow M/G-\Sigma /G$ be the projection and denote
by $M^{\ast }=M/G,$ $\Sigma ^{\ast }=\Sigma /G.$ Since $M-\Sigma $ is a
manifold, there is a diagram (not commutative in the ordinary sense) given
by the Poincar\'{e} duality:%
\begin{equation*}
\begin{array}{ccc}
H^{3}(M,\Sigma ) & \overset{\cong }{\rightarrow } & H_{1}(M-\Sigma ) \\ 
\uparrow \pi ^{\ast } &  & \downarrow \pi _{\ast } \\ 
H^{3}(M^{\ast },\Sigma ^{\ast }) & \overset{\cong }{\rightarrow } & 
H_{1}(M^{\ast }-\Sigma ^{\ast }).%
\end{array}%
\end{equation*}%
Note that $H_{1}(M-\Sigma )$ is generated by meridians to the spheres in $%
\Sigma $ and each of these is a $p$-fold cover of its image in $%
H_{1}(M^{\ast }-\Sigma ^{\ast }).$

Consider the Borel spectral sequence $E_{2}^{i,j}=H^{i}(G;H^{j}(M,\Sigma
))\Longrightarrow H^{i+j}(M^{\ast },\Sigma ^{\ast }).$ Since $\pi ^{\ast }$
factors through $H^{3}(M^{\ast },\Sigma ^{\ast })\twoheadrightarrow
E_{\infty }^{0,3}\hookrightarrow E_{2}^{0,3}=H^{3}(M,\Sigma ),$ we see that
the coker($E_{\infty }^{0,3}\hookrightarrow E_{2}^{0,3}$) is of exponent at
most $p.$ Since $H^{4}(M^{\ast },\Sigma ^{\ast })\cong \mathbb{Z}$, the
stable term $E_{\infty }^{i,j}=0$ if $i+j\geq 4,i>0.$ We consider $E_{\infty
}^{3,1}$ and $E_{\infty }^{2,2}.$ Note that $E_{2}^{3,1}$ is killed by $%
E_{3}^{0,3},$ and $E_{2}^{2,2}$ is killed by $E_{2}^{0,3}$ and $E_{2}^{4,1}.$
Therefore, 
\begin{equation*}
\mathrm{rk}E_{2}^{0,3}+\mathrm{rk}E_{2}^{4,1}\geq \mathrm{rk}E_{2}^{2,2}+%
\mathrm{rk}E_{2}^{3,1},
\end{equation*}%
which gives that $(2+L)+(3N-3)\geq 2N+2L+(N-1)$ and $L=0.$ This shows that $%
H^{2}(M)\rightarrow H^{2}(\Sigma )$ is injective and thus $\Sigma $
represents all of $H_{2}(M).$

Note that $\pi ^{\ast }$ is injective and thus $H^{3}(M^{\ast },\Sigma
^{\ast })\twoheadrightarrow E_{\infty }^{0,3}$ is isomorphic. Therefore, $%
E_{\infty }^{2,1}=0$ and $E_{2}^{2,1}$ is killed by $E_{2}^{0,2}=H^{2}(M,%
\Sigma ).$ Thus $N\geq 2(N-1),N\leq 2.$ Suppose that $N=2.$ We will prove
this is impossible to get $N=1.$ Denote by $b,c$ some two generators on $%
H^{2}(M,\Sigma )$ and $a$ a generator of $H^{1}(M,\Sigma ).$ The surjective
map $d_{2}:E_{2}^{0,2}\rightarrow E_{2}^{2,1}\ $gives that $d_{2}(b)=\alpha
a,d_{2}(c)=\beta a$ for some integers $\alpha ,\beta .$ The multiplicative
properties of the spectral sequence implies that $\ker
(d_{2}^{2,2}:E_{2}^{2,2}\rightarrow E_{2}^{4,1})=\langle \beta b-\alpha
c\rangle $ and thus there is some $e\in H^{3}(M,\Sigma )$ such that $%
d_{2}(e)=\beta b-\alpha c.$ Since $E_{3}^{3,1}=E_{2}^{3,1}\cong \mathbb{Z}%
/p\cong \langle \mu \rangle ,$ we have $d_{3}(f)=\mu a$ for some $f\in
H^{3}(M,\Sigma ),$ independent of $e.$ But $d_{3}(\alpha f)=d_{3}(\beta
f)=0, $ since $\mu \alpha a,\mu \beta b$ are in the image of $E_{2}^{1,2}.$
Therefore, $\ker (d_{2}^{2,3}:E_{2}^{2,3}\rightarrow E_{2}^{4,2})$ is of $%
\mathbb{Z}/p$-rank at least two. Since $\func{Im}(E_{2}^{0,4}\rightarrow
E_{2}^{2,3})$ is of $\mathbb{Z}/p$-rank at most $1,$ the stable term $%
E_{\infty }^{2,3}$ has $\mathbb{Z}/p$-rank at least $1.$ This is a
contradiction and thus $N=1.$
\end{proof}

\begin{theorem}
\label{main4}Suppose that $M$ is a 4-manifold with $H_{1}(M;\mathbb{Z})=0$
and the the second Betti number $b_{2}(M)\geq 3.$ Let $G$ be a minimal
non-abelian finite group of rank at least $2.$ Then $G$ cannot act
effectively homologically trivially on $M$ by homeomorphisms.
\end{theorem}

\begin{proof}
Note that $G$ contains a normal subgroup isomorphic to $(\mathbb{Z}%
/p)^{k},k\geq 2.$ Lemma \ref{elem} implies that $k=2$ and the singular set $%
\Sigma _{0}$ of $(\mathbb{Z}/p)^{k}$ is a chain of 2-spheres $%
S_{1},S_{2},\cdots ,S_{b_{2(M)}/2+1}$ arranged in a closed loop when $k=2.$
Moreover, the 2-spheres $S_{2i-1}$ (resp. $S_{2i}$) are fixed by the
generator $a$ (resp. $b$) of $\langle a\rangle \times \langle b\rangle \cong
(\mathbb{Z}/p)^{2}$. Choose an element $g\in G$ normalizing $(\mathbb{Z}%
/p)^{k}$ non-trivially. Suppose that $gag^{-1}=a^{i}b^{j}$ for some integers 
$i,j.$ Since $G$ is minimal non-abelian, the integer $j\neq 0.$ Note that
adjacent spheres $S_{i}$ and $S_{i+1}$ represent different homology classes
in $H_{2}(M)$ when $b_{2}(M)\geq 3$ by Lemma \ref{mac}. Since the action of $%
g$ on $M$ is homologically trivial and the $b_{2}(M)+2$ spheres in $\Sigma
_{0}$ represents at least $b_{2}(M)$ homology classes in $H_{2}(M)$, there
is a fixed point $x\in \Sigma _{0}\cap \mathrm{Fix}(G)$ and the element $g$
acts invariantly on each sphere $S_{i}$. However, we have that $g\mathrm{Fix}%
(a)=\mathrm{Fix}(gag^{-1})=\mathrm{Fix}(a^{i}b^{j})\neq \mathrm{Fix}(a),$ a
contradiction.
\end{proof}

\bigskip

\begin{proof}[Proof of Theorem \protect\ref{main}]
Since every non-abelian compact Lie group contains a non-abelian finite
subgroup (cf. \cite{mc} Lemma 15), Theorems \ref{main1}, \ref{main2}, \ref%
{main3}, \ref{main4} imply that the Lie group $G$ is abelian considering the
classification of minimal non-abelian finite groups. The rank of $G$ is at
most 2 by Lemma \ref{elem}.
\end{proof}

\bigskip

\section{Actions of automorphism groups of free groups}

Fixing a basis $\{a_{1},\ldots ,a_{n}\}$ for the free group $F_{n},$ we
define several elements in $\mathrm{Aut}(F_{n})$ as follows. The inversions
are defined as 
\begin{equation*}
e_{i}:a_{i}\longmapsto a_{i}^{-1},a_{j}\longmapsto a_{j}\text{ }(j\neq i);
\end{equation*}%
while the permutations are 
\begin{equation*}
(ij):a_{i}\longmapsto a_{j},a_{j}\longmapsto a_{i},a_{k}\longmapsto a_{k}%
\text{ }(k\neq i,j).
\end{equation*}%
The subgroup $N<\mathrm{Aut}(F_{n})$ generated by all $e_{i}$ $(i=1,\ldots
,n)$ is isomorphic to $(\mathbb{Z}/2)^{n}.$ The subgroup $W_{n}<\mathrm{Aut}%
(F_{n})$ is generated by $N$ and all $(ij)$ $(1\leq i\neq j\leq n).$ Denote $%
SW_{n}=W_{n}\cap \mathrm{SAut}(F_{n})$ and $SN=N\cap \mathrm{SAut}(F_{n}).$
The element $\Delta =e_{1}e_{2}\cdots e_{n}$ is central in $W_{n}$ and lies
in $\mathrm{SAut}(F_{n})$ precisely when $n$ is even.

The following result is Proposition 3.1 of \cite{bv}.

\begin{lemma}
\label{lemlast}Suppose $n\geq 3$ and let $f$ be a homomorphism from $\mathrm{%
SAut}(F_{n})$ to a group $G$. If $f|_{SW_{n}}$ has non-trivial kernel $K$,
then one of the following holds:

1. $n$ is even, $K=\langle \Delta \rangle $ and $f$ factors through $\mathrm{%
PSL}(n,\mathbb{Z})$,

2. $K=SN$ and the image of $f$ is isomorphic to $\mathrm{SL}(n,\mathbb{Z}/2)$%
, or

3. $f$ is the trivial map.
\end{lemma}

\begin{proof}[Proof of Theorem \protect\ref{th2}]
Let $G$ be the group of homologically trivial homeomorphisms of $M$ and $f:%
\mathrm{SAut}(F_{n})\rightarrow G$ a group homomorphism. Let $SW_{3}<\mathrm{%
SAut}(F_{3})$ (viewed as a subgroup of $\mathrm{SAut}(F_{n})$ fixing all $%
a_{i},$ $i>3$) be the finite group defined as above.

When $b_{2}(M)\geq 3,$ any finite subgroup of $G$ is abelian by Theorem \ref%
{main}. This implies that the kernel $K$ of $f|_{SW_{n}}$ is non-trivial. If 
$K=SN,$ the image of $f|_{\mathrm{SAut}(F_{3})}$ is isomorphic to the
non-abelian group $\mathrm{SL}(3,\mathbb{Z}/2)$ by Lemma \ref{lemlast} and
this is impossible by Theorem \ref{main4}. Therefore, the action of $\mathrm{%
SAut}(F_{3}),$ and thus that $\mathrm{SAut}(F_{n}),$ is trivial.

When $b_{2}(M)=2,$ we construct a subgroup $H=(\mathbb{Z}/3)^{2}\rtimes 
\mathbb{Z}/2<$ $\mathrm{SAut}(F_{4})$ (viewed as a subgroup of $\mathrm{SAut}%
(F_{n})$ fixing all $a_{i},$ $i>4$), where the generator $b$ of $\mathbb{Z}%
/2 $ permutes the two generators of $(\mathbb{Z}/3)^{2}.$ When this is done,
Lemma \ref{mac} implies that the singular set $\Sigma $ of $(\mathbb{Z}%
/3)^{2}$ is a chain of 2-spheres arranged in a closed loop. Moreover, each
sphere represents a primitive class in $H_{2}(M;\mathbb{Z})$ and together
these classes generate $H_{2}(M;\mathbb{Z}).$ However, the element $b$
permutes spheres in the singular set $\Sigma ,$ which is impossible since
the group action of $b$ is homologically trivial. This shows that the action
of $H$ is not effective and the theorem is proved by Lemma \ref{lemlast} and
Theorem \ref{main}. Actually, for $i=1,2$, we define $R_{i}:F_{n}\rightarrow
F_{n}$ as 
\begin{equation*}
a_{2i-1}\longmapsto a_{2i}^{-1},a_{2i}\longmapsto
a_{2i}^{-1}a_{2i-1},a_{j}\longmapsto a_{j}(j\neq 2i,2i-1).
\end{equation*}%
The subgroup generated by $R_{1},R_{2}$ is isomorphic to $(\mathbb{Z}/3)^{2}$
(cf. Bridson-Vogtmann \cite{bv}, Lemma 3.2)$.$ The generator $R_{1}$ is
conjugate to $R_{2}$ by the involution $(13)(24),$ giving the generator of $%
\mathbb{Z}/2.$

When $b_{2}(M)=1,$ we consider the subgroup $N\cap \mathrm{SAut}(F_{4}),$
which is isomorphic to $(\mathbb{Z}/2)^{3}$ whose generators will be denoted
by $a,b,c.$ Suppose that $(\mathbb{Z}/2)^{3}$ acts effectively. Lemma \ref%
{mac} implies that the singular set $\Sigma $ of $a,b$ is a loop of $%
b_{2}(M)+2=3$ 2-spheres, intersecting at three global fixed points $x,y,z\in 
\mathrm{Fix}(a,b).$ But $c$ acts invariantly on each sphere in $\Sigma .$
The Borel formula \ref{borel} gives that the action of $c$ is non-trivial.
Therefore, the fixed point set of any index-2 subgroup of $(\mathbb{Z}%
/2)^{3} $ is discrete, which is impossible by Borel's formula again. This
implies that some non-trivial element $g\in (\mathbb{Z}/2)^{3}$ acts
trivially. Lemma \ref{lemlast} implies that $\func{Im}f$ is either trivial
or contains a non-abelian subgroup. Considering Theorem \ref{main}, the
image $\func{Im}f $ has to be trivial.
\end{proof}

\bigskip

\bigskip

\noindent \textbf{Acknowledgements}

The author is grateful to the referee for detailed comments on a previous
version of this article. This work is supported by NSFC (No. 11971389).

\bigskip 

NYU Shanghai, 1555 Century Avenue, Shanghai, 200122, China.

NYU-ECNU Institute of Mathematical Sciences at NYU Shanghai, 3663 Zhongshan
Road North, Shanghai, 200062, China

E-mail: sy55@nyu.edu

\end{document}